\documentclass[english,12pt]{smfart}
\usepackage{amssymb}
\usepackage{amsfonts}
\usepackage{amscd}
\usepackage[all]{xypic}

\CompileMatrices

\swapnumbers

\def\ord{{\rm ord}}

\def\SP{{\rm SP}}

\def\11{{\mathbf 1}}
\def\AA{{\mathbf A}}

\def\CC{{\mathbf C}}

\def\FF{{\mathbf F}}

\def\NN{{\mathbf N}}

\def\QQ{{\mathbf Q}}
\def\RR{{\mathbf R}}

\def\ZZ{{\mathbf Z}}

\def\cM{{\mathcal M}}

\def\cO{{\mathcal O}}

\def\lla{\mathopen{\{\!\{}}
\def\rra{\mathopen{\}\!\}}}

\mathchardef\alphag="7C0B \mathchardef\betag="7C0C
\mathchardef\gammag="7C0D \mathchardef\deltag="7C0E
\mathchardef\varepsilong="7C22 \mathchardef\varphig="7C27
\mathchardef\psig="7C20 \mathchardef\zetag="7C10
\mathchardef\epsilong="7C0F \mathchardef\rhog="7C1A
\mathchardef\taug="7C1C \mathchardef\upsilong="7C1D
\mathchardef\iotag="7C13 \mathchardef\thetag="7C12
\mathchardef\pig="7C19 \mathchardef\sigmag="7C1B
\mathchardef\etag="7C11 \mathchardef\omegag="7C21
\mathchardef\kappag="7C14 \mathchardef\lambdag="7C15
\mathchardef\mug="7C16 \mathchardef\xig="7C18
\mathchardef\chig="7C1F \mathchardef\nug="7C17
\mathchardef\varthetag="7C23 \mathchardef\varpig="7C24
\mathchardef\varrhog="7C25 \mathchardef\varsigmag="7C26
\mathchardef\Omegag="7C0A \mathchardef\Thetag="7C02
\mathchardef\Sigmag="7C06 \mathchardef\Deltag="7C01
\mathchardef\Phig="7C08 \mathchardef\Gammag="7C00
\mathchardef\Psig="7C09 \mathchardef\Lambdag="7C03
\mathchardef\Xig="7C04 \mathchardef\Pig="7C05
\mathchardef\Upsilong="7C07

\newtheorem{theorem}[subsection]{Theorem}
\newtheorem{lem}[subsection]{Lemma}
\newtheorem{cor}[subsection]{Corollary}
\newtheorem{prop}[subsection]{Proposition}

\theoremstyle{definition}
\newtheorem{definition}[subsection]{Definition}
\newtheorem{example}[subsection]{Example}

\newtheorem{def-prop}[subsection]{Proposition-Definition}
\newtheorem{def-theorem}[subsection]{Theorem-Definition}
\newtheorem{def-lem}[subsection]{Lemma-Definition}

\theoremstyle{remark}
\newtheorem{remark}[subsection]{Remark}

\theoremstyle{plain}

\numberwithin{equation}{subsection}

\def\boxit#1#2{\setbox1=\hbox{\kern#1{#2}\kern#1}%
\dimen1=\ht1 \advance\dimen1 by #1 \dimen2=\dp1 \advance\dimen2 by
#1
\setbox1=\hbox{\vrule height\dimen1 depth\dimen2\box1\vrule}%
\setbox1=\vbox{\hrule\box1\hrule}%
\advance\dimen1 by .4pt \ht1=\dimen1 \advance\dimen2 by .4pt
\dp1=\dimen2 \box1\relax}

\renewcommand{\theequation}{\thesubsection.\arabic{equation}}

\def\AA{{\mathbf A}}

\def\CC{{\mathbf C}}

\def\FF{{\mathbf F}}

\def\NN{{\mathbf N}}

\def\QQ{{\mathbf Q}}
\def\RR{{\mathbf R}}

\def\ZZ{{\mathbf Z}}

\def\cM{{\mathcal M}}

\def\cO{{\mathcal O}}

\mathchardef\alphag="7C0B \mathchardef\betag="7C0C
\mathchardef\gammag="7C0D \mathchardef\deltag="7C0E
\mathchardef\varepsilong="7C22 \mathchardef\varphig="7C27
\mathchardef\psig="7C20 \mathchardef\zetag="7C10
\mathchardef\epsilong="7C0F \mathchardef\rhog="7C1A
\mathchardef\taug="7C1C \mathchardef\upsilong="7C1D
\mathchardef\iotag="7C13 \mathchardef\thetag="7C12
\mathchardef\pig="7C19 \mathchardef\sigmag="7C1B
\mathchardef\etag="7C11 \mathchardef\omegag="7C21
\mathchardef\kappag="7C14 \mathchardef\lambdag="7C15
\mathchardef\mug="7C16 \mathchardef\xig="7C18
\mathchardef\chig="7C1F \mathchardef\nug="7C17
\mathchardef\varthetag="7C23 \mathchardef\varpig="7C24
\mathchardef\varrhog="7C25 \mathchardef\varsigmag="7C26
\mathchardef\Omegag="7C0A \mathchardef\Thetag="7C02
\mathchardef\Sigmag="7C06 \mathchardef\Deltag="7C01
\mathchardef\Phig="7C08 \mathchardef\Gammag="7C00
\mathchardef\Psig="7C09 \mathchardef\Lambdag="7C03
\mathchardef\Xig="7C04 \mathchardef\Pig="7C05
\mathchardef\Upsilong="7C07

\def\ord{{\rm ord}}

\begin{document}

\author{Raf Cluckers}
\address{Universit\'e Lille 1, Laboratoire Painlev\'e, CNRS - UMR 8524, Cit\'e Scientifique, 59655
Villeneuve d'Ascq C\'edex, France, and,
Katholieke Universiteit Leuven, Department of Mathematics,
Celestijnenlaan 200B, B-3001 Leu\-ven, Bel\-gium\\  }
\email{raf.cluckers@wis.kuleuven.be}
\urladdr{http://www.wis.kuleuven.be/algebra/Raf/}

%\author[possibly others]{possibly other authors, or included in a larger joint paper}

%\address{Laboratoire J.-A. Dieudonn\'e,
%Universit\'e de Nice - Sophia Antipolis, Parc Valrose, 06108 Nice
%Cedex 02, France (UMR 6621 du CNRS)} \email{comte@math.unice.fr}
%\urladdr{http://www-math.unice.fr/membres/comte.html}

\title[$p$-adic van der Corput Lemma]{Analytic van der Corput Lemma for $p$-adic and $\FF_q((t))$ oscillatory integrals, singular {F}ourier transforms, and restriction theorems}

\begin{abstract}
We give the $p$-adic and $\FF_q((t))$ analogue of the real van der Corput Lemma, where the real condition of sufficient smoothness for the phase is replaced by the condition that the phase is a convergent power series.
This van der Corput style result allows us, in analogy to the real situation, to study singular Fourier transforms on suitably curved (analytic) manifolds and opens the way for further applications. As one such further application we give the restriction theorem for Fourier transforms of $L^p$ functions to suitably curved analytic manifolds over non-archimedean local fields, similar to the real restriction result by E.~Stein and C.~Fefferman.
%In the more special case that the phase is a polynomial, K.~Rogers had already obtained a  $p$-adic analogue of a variant of the real van der Corput %Lemma.
\end{abstract}

\keywords{van der Corput Lemma, $p$-adic oscillatory integrals of the first kind, singular Fourier transforms, $p$-adic restriction theorems, $p$-adic analytic manifolds of finite type}

\maketitle

\section{Introduction}

An important part of the theory of harmonic analysis on abelian locally compact groups can be developed in parallel for all these groups. Great effort has gone in making many results in harmonic analysis, often first discovered on real affine spaces and on the circle group, as uniform as possible in the abelian locally compact group,   starting around the study by A.~Weil \cite{Weilgroup}. The additive groups of local fields (be it archimedean or non-archimedean), together with affine spaces over these fields, are usually the base cases for which results and proofs should be more or less uniform in the field. Some branches of the theory seem to be tight up more closely with the real and complex numbers, like for example Hardy spaces or, more classically, harmonic functions. In E.~Stein's book \cite{Stein}, although the set-up is related to real-variable methods, many results translate with some care to results over $p$-adic fields, in particular, several results about maximal functions that are related to the doubling nature of the measure, see
\cite[Chapters I and II]{Stein}, translate often directly.
 However, real results by van der Corput going back as far as 1921 (see \cite{vdCorput} or \cite{Stein}, Chapter VIII, Proposition 2) on one-dimensional one-parameter oscillatory real integrals were not known to have a counterpart over non-archimedean local fields, and many results of \cite{Stein}, from chapter VIII on,
 are based on this van der Corput Lemma, which has moreover been applied in a wide variety of (real) settings. In this paper we give the non-archimedean analogue of the van der Corput Lemma.
One has to note however that a literal analogue of the van der Corput Lemma to non-archimedean local fields is false because of several reasons: firstly, the constants are not absolute but depend in particular on the local field; secondly and more fundamentally, the condition of being $C^k$ for the phase is much too general a condition for a function from a non-archimedean local field $K$ to $K$ to be able to derive anything nontrivial (by the total disconnectedness).
Let us recall van der Corput's Lemma on real oscillatory integrals of \cite{vdCorput} in the form of \cite{Stein}, Chapter VIII, Proposition 2, where $f$ is a real-valued $C^k$ function on an open interval $(a,b)$ such that $|f^{(k)}(x)|\geq 1$ for some $k\geq 1$ and all $x$ in $(a,b)$. If either $k\geq 2$ or $k=1$ and $f'$ is monotonic, then
$$
R(y):=\int_a^b \exp( 2\pi i y\cdot f(x)) dx
$$
satisfies
$$
|R(y)|\leq c_k |y|^{-1/k}\quad \mbox{ for all nonzero } y,
$$
where $c_k$ is a constant only depending on $k$ (and thus not on $a,b,y$, nor on $f$).

In Proposition \ref{osc(k)} below we give the $p$-adic and $\FF_q((t))$ analogue for analytic phase $f$
of this van der Corput Lemma, allowing us to develop the theory further in great analogy to Chapter VIII of \cite{Stein}. In particular we are able to study the relation between $K$-analytic manifolds with suitable curvature (namely manifolds of finite type) and Fourier transforms, see Theorem \ref{thm:Fourier:finitetype} and the Restriction Theorem \ref{thm:restr}. Note that the constants $c_k$ that we will get for non-archimedean $K$ will depend on the field $K$ as well as on the Gauss norm of the analytic phase. Luckily enough such Gauss norms are bounded in many natural situations so that we will encounter no difficulty in proceeding to the study of Fourier transforms in higher dimensions.

In essense,  van der Corput's Lemma on the reals is based on the fundamental theorem of calculus relating integrals with derivatives, namely on its basic corollary that for a real $C^1$ function
$\phi:\RR\to\RR$, if $\phi(c)=0$ and $|\phi'(x)|\geq \varepsilon>0$
on $\RR$, then $|\phi(x+c)|\geq \varepsilon |x|$ for all $x\in\RR$. Such a fundamental theorem does not have an analogue over non-archimedean local fields, but if $\phi$ is the identity function $x\mapsto x$, its corollary trivially holds over $K$, and one might try to apply a change of variables to reduce to the identity function in general. However, $C^1$-functions $K\to K$ do not allow an analogue of the real implicit function theorem, so such functions seem hopeless. We resolve this problem by requiring that the phase $f$ of the oscillatory integral be $K$-analytic, and such functions clearly allow implicit function theorems.

In Stein's version of the proof of van der Corput's Lemma one divides the interval $(a,b)$ into at most three sub-intervals: a small, bad interval where a trivial bound is used, and the remaining two larger and nice intervals where one can use induction on $k$. Later on, the size  of the bad interval is optimized to find the desired bounds.
A difficulty in adapting Stein's version of the proof of van der Corput's Lemma is that, while cutting away one bad sub-interval of $(a,b)$ one is left with at most two remaining intervals in the real case, in the non-archimedean case if one cuts away a small (bad) ball out of a big ball, one is left with a possibly huge (but still finite) number of remaining sub-balls. Hence, one has to control not only what size of balls the induction hypothesis can be applied to, but also the number of balls in which one subdivides the bigger ball, before optimizing the size of the bad ball on which the trivial bound is used.

 \subsection*{}
 In the more special case that the phase $f$ is a polynomial over $K$, K.~Rogers \cite{Rogers} obtained a  $p$-adic analogue of a variant of the real van der Corput Lemma (see Corollary \ref{general:Rogers} below for a $p$-adic generalization of this variant), but  to develop the theory further as we do, one really seems to need results with analytic phase in the oscillating integrals. Indeed, the ability of having $K$-analytic charts on manifolds is much more flexible and general for applications than having a polynomial framework. While the real van der Corput Lemma is already quite old, the applications we give in Theorem \ref{thm:Fourier:finitetype} and Theorem \ref{thm:restr} are  non-archimedean analogues of much more recent real results, see \cite{Stein}, Chapter VIII. Note that in the real case K.~Rogers \cite{Rogerssharp} gives very good estimates for the constants $c_k$ for the above stated real van der Corput Lemma.

The study presented in this paper arose in the context of the study of groups with the Howe - Moore property  in \cite{CLTV}. Theorem \ref{thm:Fourier:finitetype} is used in \cite{CLTV} to give an alternative proof for the Howe  - Moore vanishing theorem in the $p$-adic case. We would like to thank warmly A.~Valette for
inviting us cordially to work on the question addressed in Theorem \ref{thm:Fourier:finitetype}. Further we thank K.~Rogers for inspiring us to study the relation of our results to his beautiful work in \cite{Rogers}, which led us to formulate Corollary \ref{general:Rogers}.

\section{Preliminaries}\label{sec:preli}

Write $K$ for a fixed non-archimedean local field and $\cO_K$ for its
valuation ring with maximal ideal $\cM_K$.  Let $q_K=p_K^{e_K}$ be
the number of elements of the residue field $\cO_K/\cM_K$, where
$p_K$ is a prime number and $e_K\geq 1$. Write $\pi_K$ for a uniformizer of $\cO_K$ and fix the norm $|\cdot|$ on
$K$ by assigning the value $q_K^{-1}$ to $\pi_K$, and write $\ord:K \to\ZZ \cup\{+\infty\}$ for the order which assigns the value $1$ to $\pi_K$ and sends $0$ to $+\infty$. For $x$ in
$K^n$, $|x|$ stands for $\max_{i=1}^n|x_i|$. Let $\psi$ be an
additive character on $K$ which is trivial on $\cM_K$ and nontrivial
on $\cO_K$.

\subsection{Convergent and special power series}

For $x$ a variable, resp.~a tuple of variables $(x_1,\ldots,x_n)$, write $K \lla x\rra$ for the collection of power series in $x$ over $K$ which converge on $\cO_K$, resp.~on $\cO_K^n$, that is, those
power series $\sum_{i\in\NN^n} a_ix^i\in K[[x]]$ satisfying that
$|a_i|$ goes to zero when $|i|:=i_1+\ldots,+i_n$ goes to infinity. Likewise, write
$\cO_K \lla x\rra$ for power series in $K \lla x\rra$ which also lie
in $\cO_K[[x]]$. For $f(x)\in K\lla x \rra $, write $\| f\| $ for the
Gauss norm of $f$, which is by definition $\sup_{i}|a_i|$. From now on untill Section \ref{sec:sev}, $x$ will always denote one variable. \\

The following definition of Special Power series, abbreviated by  SP, is a one-variable $p$-adic and $\FF_q((t))$ analogue of real $C^1$ functions $(a,b)\to \RR$ with big derivative on a real interval $(a,b)$.

\begin{definition}\label{defSP}
A power series $\sum_{i\geq
0} a_ix^i$ in one variable is called SP if it lies in $K\lla x\rra $,  $a_1\not=0$, and
  $a_j\in a_1\cM_K$ for all $j>1$. If $f$ is SP, write $|f|_{\SP}$ for $|a_1|$, which is nothing else than the Gauss norm of
  $f-f(0)$.
\end{definition}

Note that a convergent power series $f=\sum_{i\geq
0} a_ix^i$ is SP if and only if the higher order terms have small coefficients compared to the linear term in the sense that $|a_j|<|a_1|$ for each $j>1$. Therefore, $f$ can be approximated by $a_0+a_1x$ in the senses that  for all $x\in \cO_K$
$$
 |f(x)-a_0-a_1 x|< |f(x)|,
$$
and
 $$
 |f'(x)|=|a_1|.
$$

Although the definition of SP may seem very restrictive, the philosophy behind it is that power
series often become SP after basic manipulations like zooming in to good parts of the domain or taking derivatives. Lemma \ref{1SP} and Lemma \ref{d(d-1)} exhibit this kind of phenomena. Lemma \ref{SP2alt} describes the linear behavior of $|f(x)|$ in more detail.

% To determine how far to zoom in in the domain before a power series $f=\sum a_i x^i$ becomes SP, it is useful to compare $|a_1|$  with $||f_{\geq 2}||$, where  %$f_{\geq 2} = \sum_{i\geq 2} a_i x^i$.
\begin{definition}\label{defSPn}
Let $f(x)$ be in  $K\lla x\rra $. Define the SP-number of $f$ as the smallest integer $r\geq 0$ such that for all nonzero $c\in\cM_K^r$ and all $b\in \cO_K$, the power series
$$
f_{b,c}(t):=
\frac{1}{c}f(b+ct)
$$
is SP if such $r$ exists, and define the SP-number of $f$ as $+\infty$ otherwise.
\end{definition}

\begin{lem}\label{1SP}
Let $f(x)=\sum_{i\geq 0}a_i x^i$ be in $K \lla x\rra $. % and write $f_{\geq 2}(x)$ for $\sum_{i\geq 2}a_i x^i$. % $= f(x) - a_0 - a_1 x$.
Suppose that  $|f'(x)|\geq 1$ for all $x$ in
$\cO_K$. Then the SP-number of $f$ is an integer $r$ satisfying
$$
q_K^{r-1}\leq \|f-f(0)\|.
$$
 Moreover, for all nonzero $c\in\cM_K^r$ and all $b\in \cO_K$, one has $|f_{b,c}|_{SP} \geq 1$.

\end{lem}
\begin{proof}
Note that $\|f-f(0)\|\geq 1$. If $f$ is already SP the statement is clear. Namely, $f$ is SP if and only if its SP-number is $0$. Now suppose that $f$ is not SP. For nonzero $c$ in $\cO_K$ such that $|c| < \|  f - f(0) \| ^{-1}$ and for $b\in\cO_K$,
%Since $\cO_K$ is compact, we may suppose that $1$ equals the minimal value of $|f'(x)|$ for $x$ in
%$\cO_K$ (otherwise, replace $f$ by a multiple of $f$).
expand the power series $f_{b,c}(t)$ in $t$ as
$$
f_{b,c}(t) = \sum_{j\geq 0} b_{j}t^j.
$$
The chain rule for differentiation implies that $|b_{1} | =  |(f_{b,c})'(0)| = |f'(b)| \geq 1$.
 On the other hand, each $b_{j}$ for $j>1$ lies in $\cM_K$ by the above choice of $c$. Concluding, $f_{b,c}(t)$ is SP and $|f_{b,c}|_{SP} = |b_{1} | \geq 1$.
\end{proof}

\begin{example}\label{ex:SPGauss}
Clearly a power series in $K\lla x\rra $ is SP if and only if it has SP-number $0$.
Let $f(x)$ be the polynomial $x+ 2^{-k}x^{2^{k+1}}$ for some integer $k\geq 0$ and suppose that $K=\QQ_2$, the field of $2$-adic numbers. Then $|f'(x)|=1 $ for all $x\in\ZZ_2$, and although the Gauss norm of $f-f(0)$ is big, the SP-number of $f$ is just $1$.

%For a second example of a different nature, let $f(x)$ be $x+ x^{3}+ x^5$ and let $K$ be again $\QQ_2$.
%Then $|f'(x)|= |1 +  3 x^{2} + 5 x^{4}|= 1 $ for any $x\in \ZZ_2$.

%For a third example, for an integer $\ell >0$, let $f_{\ell}(x)$ be $$\frac{x^\ell+ x^{\ell+2} %}{(\ell+2)!}$$ and let $K$ be again $\QQ_2$. Then $|f_{\ell,n}^{(\ell + 2)}| $  equals $1$ on $\ZZ_2$ .
\end{example}
%\begin{remark}
%It might be possible to sharpen the conclusion of Lemma \ref{1SP} in two possible ways in more concrete situations: for particular $b$, $|f_{b,c}|_{SP}$ might be strictly bigger than $1$, or, for each $b$ separately, a largest constant $c_b$ (that is, with maximal norm) can be chosen such that $f_{b,c_b}$ is SP. In the second case, a partition of $\cO_K$ into balls $b+c_b\cO_K$ might be obtainable with fewer parts than using a fixed $c$ for all $b$ and a partition with balls of fixed radius of the form $b+c\cO_K$.
%\end{remark}

%We will use Lemma \ref{1SP} in the form of the following corollary.
%\begin{cor}\label{cor:1SP}
%Let $f(x)=\sum_{i\geq 0}a_i x^i$ be in $K \lla x\rra $.
% Suppose that  $|f^{(k)}(x)|\geq 1$ for all $x$ in
%$\cO_K$ and some $k\geq 1$, where $f^{(k)}$ is the $k$-th derivative of $f(x)$. If $f^{(k-1)}$ is SP then let $c$ be arbitrary in $\cO_K\setminus \{0\}$. If $f^{(k-1)}$ is not SP, then
%let $c\in \cO_K\setminus \{0\}$ satisfy
%$$
%|c| < \|  f - f(0) \| ^{-1}.% ^{-\frac{1}{2}}.
%$$
% Then, in all cases and for any $b\in \cO_K$,  $f_{b,c,k}^{(k-1)}$ is SP and satisfies $|f^{(k-1)}_{b,c,k}|_{SP} \geq 1$, with
%$$
%f_{b,c,k}(t):=
%\frac{1}{c^k}f(b+ct).
%$$
%\end{cor}
%\begin{proof}
%Note that, for any $k\geq 0$,
%$$
%\|  f^{(k-1)} - f^{(k-1)}(0) \|  \leq \|  f - f(0) \| .
%$$
%Now either one proceeds as in the proof of Lemma \ref{1SP}, or one applies Lemma \ref{1SP} and the chain rule for differentiation to the power series $f^{(k-1)}$.
%\end{proof}

\begin{definition}\label{def:reg}

Call $f(x)$ in $\cO_K \lla x\rra $ (Weierstrass) regular of degree $d\geq 0$ if
$f(x)$ is congruent to a monic polynomial of degree $d$ modulo the
ideal $\pi_K\cdot \cO_K \lla x\rra$.

\end{definition}

It is clear that for any nonzero $f\in K\lla x\rra $, one has for a
unique $d\geq 0$ and a unique $c\in K^\times $ that $cf$ is regular
of degree $d$.

\begin{lem}\label{d(d-1)}%[Char $K=0$]
\label{somek} Let $f(x)\in K \lla x\rra $ be nonconstant. Let
$c$ be the unique element of $K^\times$ such that $c\cdot
(f(x)-f(0))$ is regular of degree $d\geq 1$.  If the characteristic of $K$ is zero or larger than $d$, then $f^{(d-1)}$ is SP and $|f^{(d-1)}|_{SP} = |d!|\cdot |c|^{-1}$. Moreover, the condition on the characteristic  of $K$ is necessary.
\end{lem}
\begin{proof}
Clearly the condition on the characteristic of $K$ is necessary. Now suppose that the characteristic of $K$ is zero or  $>d$ and write $f=\sum_{i\geq 0} a_ix^i$.
 %and $c= a_d^{-1}$.
The coefficient of the linear term of $f^{(d-1)}$ equals $d! a_d$ with $c^{-1}=a_d$, and is thus nonzero. For any $j>1$, the $j$th coefficient of $f^{(d-1)}$ equals
$$
a_{j+d-1} \prod_{i=1}^{d-1} (j+d-i).
$$
Since for any $j\geq 1$ one has
$$
| \prod_{i=1}^{d-1} (j+d-i) | \leq  |d!|,
$$
and since $d$ equals the maximum of all integers $j$ such that $|a_j| = \|  f(x)-f(0) \| $,
it follows that $f^{(d-1)}$ is SP.

\end{proof}

The following lemma gives a link between $f$ being SP and a lower
bound for $|f(x)|$. % for $x$ away from some small sub-ball in $\cO_K$.
It is the analogue of the fact that for a real $C^1$ function
$\phi:\RR\to\RR$, if $\phi(c)=0$ and $|\phi'(x)|\geq \varepsilon>0$
on $\RR$, then $|\phi(x+c)|\geq \varepsilon |x|$ for all $x\in\RR$.
In the real case this follows of course from the Fundamental Theorem
of integral calculus, but on $K$ one has to proceed
differently.

\begin{lem}\label{SP2alt}%[NEW]
 Let $f=\sum_{i\geq 0 }a_i x^i$ be in $K\lla
x\rra$. Suppose that $f$ is SP. If there exists $d\in \cO_K$ such
that $f(d)=0$, then
$$
|f (x) | = |f|_{\SP} \cdot | x - d |
$$
 for
all $x\in \cO_K$. If there exists no such $d$, then
$$
|f (x) | = |a_0| >  |f|_{\SP}
$$
for all $x\in \cO_K$. In general, if $e\in\cO_K$ is such that
$|f(e)|$ is minimal among the values $|f(x)|$ for $x$ in $\cO_K$, then one
has for all $x\in\cO_K$
$$
|f (x) | \geq  |f|_{\SP} \cdot | x - e |.
$$
\end{lem}
\begin{proof}
Clearly $|a_0| \leq   |f|_{\SP}$ if and only if there exists
$d\in\cO_K$ such that $f(d)=0$ (this follows for example from non-archimedean Weierstrass preparation). If $|a_0| >  |f|_{\SP}$ then clearly
$ |f (x) | = |a_0|$ for all $x\in \cO_K$ and this finishes the
second case. If $f(d)=0$ for some $d\in\cO_K$, then, with $g(t)=f(t+d)$, one has that
$g(0)=0$ and that $g$ is SP, which implies that
$|g(t)|=|g|_{\SP}|t|$ for all $t\in\cO_K$. This finishes the first
case since $|g|_{\SP}=|f|_{\SP}$. For the final statement, in the second case, any $e\in
\cO_K$ can serve; in the first case, one has to take $e=d$.
\end{proof}

\begin{cor}\label{SPff'}
Suppose that $K$ has characteristic zero.
 Let $f=\sum_{i\geq 0 }a_i x^i$ be in $K\lla
x\rra$. Suppose that $f'$ is SP and that $|f'(x)|>0$ for all $x\in\cO_K$. Then the SP-number of $f$ is at most equal to $\ord(p_K)$, the ramification degree of $K$.
\end{cor}
\begin{proof}
By the second case of Lemma \ref{SP2alt}, we find $|a_1| >|2 a_2| >|ja_j|$ for all $j\geq 3$. Hence, if we take for $r$ the smallest integer satisfying $r>\ord(p_K)/p_K$, any nonzero $c\in \cM_K^r$, and any $b\in\cO_K$, then $f_{b,c}(t)$
is SP, as one can see by expanding in $t$.
\end{proof}
In fact, the above proof of Lemma \ref{SPff'} yields the stronger bound $ \lfloor \frac{\ord(p_K)}{p_K}+1 \rfloor$ %or  $\lceil \ord(p_K)/ (p_K-1) \rceil $
for the SP-number of $f$.

\section{Oscillatory integrals}

We present $p$-adic and $\FF_q((t))$ analogues of Chapter VIII of \cite{Stein}, namely of what E.~Stein calls the theory of oscillatory integrals of the first kind.
 We motivate some of our choices for the possible reader with a better background in the  real setting than in the non-archimedean setting.
 For an oscillatory integral (of the first kind), typically of the form
$$
I(y) = \int_{\cO_K} \psi (y\cdot f(x)) g(x) |dx|,
$$
 where $\psi$ is the additive character on $K$ as introduced at the beginning of section \ref{sec:preli} and $|dx|$ is the Haar measure on $K$ normalized so that $\cO_K$ has measure $1$, the function $f$ is usually called the phase and $g$ the amplitude of the integral. For the  many variables analogue, $x$ or $y$ can be tuples of variables and $f$ can be a tuple of $K$-valued functions, and then $y\cdot f$ is the standard inner product.

In the non-archimedean set-up, $f$ takes values in $K$ while $g$ takes real or complex values. While in Stein's set-up $f$ and $g$ are usually assumed to be sufficiently smooth in the sense of sufficiently continuously differentiable, we will have to make choices on which functions $f$ and $g$ to focus: $C^\infty$ conditions on $f$ are too general because of the total disconnectedness of $K$ (and the implicit function theorem can fail for $C^k$ functions $K\to K$). We typically require that $f$ is given by a convergent power series. One usually requires that $g:\cO_K\to \CC$ is $C^{\infty}$. Any $C^\infty$ function $g:\cO_K\to \CC$ is locally constant, and by the compactness of $\cO_K$ it has finite image. Therefore, we will assume that $g$ is constantly equal to $1$; any $C^\infty$ function can be brought back to this situation by taking finite partitions, scaling the parts by homotheties, and replacing $g$ by a multiple.

By similar scaling arguments, one can usually reduce integrals over more general domains to integrals over $\cO_K$ (or over Cartesian powers of $\cO_K$), and conditions of the form $|f'(x)|\geq \varepsilon$ can be reduced to the more simple condition $|f_1'(x)|\geq 1$ where $f_1$ is a multiple of $f$. Hence several of the statements below, like e.g.~the van der Corput style Proposition \ref{osc(k)}, are more general than they seem at first sight.

\subsection{The one variable theory}

%As the theory of asymptotics for oscillating integrals over $K$ is classical

We first state an almost trivial variant of classically known results, Lemma \ref{oscSP}, about arbitrarily quick decays at infinity if the phase of the oscillatory integral is nice enough, where in our set-up nice enough means SP and quick decay actually means identically zero for large $y$.

\begin{lem}\label{oscSP}
Let $f(x)=\sum_{i\geq 0} a_ix^i$ in $K \lla x\rra$ be SP.  Then, for all $y\in K$ with $|y|\geq |a_1|^{-1}$ one has
$$
\int_{\cO_K} \psi (y\cdot f(x))|dx| = 0
$$
and,  for $y$ with $|y|< |a_1|^{-1}$ one has
$$
\int_{\cO_K} \psi (y\cdot f(x))|dx|= \psi(y\cdot a_0).
$$
Combining, one has
%$$
%|\  \int_{\cO_K} \psi (y\cdot f(x))|dx| \ | \sim  o( |y|^{-N} )
%$$
%for each $N>0$ and
$$
|\ \int_{\cO_K} \psi (y\cdot f(x))|dx| \ | \leq  q_K^{-1} |a_1|^{-1}
|y|^{-1} \mbox{ for all nonzero } y.
$$
\end{lem}
\begin{proof}
There is no loss in replacing $f$ by a multiple so that one has
$|a_1|=1$. The equalities follow from the fact that $\psi$ is
trivial on $\cM_K$ and nontrivial on $\cO_K$, and from the basic
relation of character sums (namely, for a nontrivial
character $\omega$ on a finite abelian group $G$, the sum  $\sum_{g\in G}\omega(g)$ equals
zero). The summarizing statement follows from the fact that the norm of $\int_{\cO_K} \psi (y\cdot f(x))|dx|$ is always $\leq 1$ and that if  $|y|< |a_1|^{-1}$ then $q_K^{-1} |a_1|^{-1} |y|^{-1}\geq 1$.
\end{proof}

\subsubsection*{Van der Corput's Lemma}
%Van der Corput's Lemma on real oscillatory integrals \cite{vdCorput} states that if $f$ is real-valued and sufficiently differentiable on an open interval $(a,b)$, and if $|f^{(k)}(x)|\geq 1$ for some $k\geq 2$ and all $x$ in $(a,b)$, then
%$$
%R(y):=\int_a^b \exp( 2\pi i y\cdot f(x)) dx
%$$
%satisfies
%$$
%|R(y)|\leq c_k |y|^{-1/k}\quad \mbox{ for all nonzero } y,
%$$
%where $c_k$ is a constant only depending on $k$ (and thus not on $a,b,y$, nor on $f$). For $k=1$, more conditions on $f$ are needed to get bounds of the same form, namely, $f'$ should moreover be monotonic on $(a,b)$, see E.~Stein
%\cite{Stein}, Chapter VIII, Proposition 2. In \cite{Rogerssharp}, K. Rogers gives good estimates for the constants $c_k$, still in the real case.
% In Proposition \ref{osc(k)} below we give the $p$-adic and $\FF_q((t))$ analog for analytic phase $f$
%of this van der Corput Lemma, where the real condition of sufficient smoothness for the phase is replaced by the condition that $f$ is in $K\lla x\rra$.
% In the more special case that $f$ is a polynomial over $K$, K.~Rogers \cite{Rogers} had already obtained a  $p$-adic analogue of a variant of the real van der Corput Lemma. The methods of \cite{Rogers} seem not directly transposable to the case of analytic phases $f$.

Fix $f(x)=\sum_{i\geq 0} a_ix^i$ in $K \lla x\rra$ and write, for $y\in K$,
$$
I(y)=  \int_{\cO_K} \psi (y\cdot f(t))|dt|.
$$
Note that $I(y)$ is the non-archimedean analogue of the real integral $R(y)$ of the introduction.

\begin{prop}[Analytic, non-archimedean van der Corput Lemma]\label{osc(k)}%[NEWER]
Suppose that
for some $k\geq 1$ one has that $|f^{(k)}(x)|\geq 1$ for all $x$ in
$\cO_K$. Then one has for all $y\in K^\times$
$$
| I(y) | \leq c_k
|y|^{-\frac{1}{k}},
$$
where $c_k$ only depends on $k$,  $q_K$, and on the Gauss norm of $f-f(0)$. Alternatively, if $K$ has characteristic zero, then $c_k$ can be taken only depending on $k$, $q_K$, the ramification degree $\ord(p_K)$ of $K$, and on the SP-number of $f^{(k-1)}$.
% If $K$ has positive characteristic, then $c_k$ only depends on $K$ and on the Gauss norm of $f-f(0)$.

 %One can take the $c_k$ growing in $\| f^{(k)} \| $.  %If moreover $f^{(k-1)}$ is SP, then $c_k$ can be taken depending on $k$ and $q_K$ only.
\end{prop}
%\begin{remark}
%By Lemma \ref{1SP} if one knows a bound on the Gauss norm of $f-f(0)$, then automatically the SP-numbers of all derivatives
%\end{remark}

\begin{proof}%[Proof of Proposition \ref{osc(k)}]

%Since $||f-f(0)||\geq 1$ and since the Gauss norm takes discrete values,
%It is clear that the $c_k$ can be taken growing as $\| f^{(k)}\| \geq 1$ grows.

If $|y|< 1$ then in fact any $c_{k}\geq q_K^{-1}$ can do in the bound for $|I(y)|$. Hence, we may suppose that $|y|\geq 1$.
First we work for general $k$. By Lemma \ref{1SP} the SP-number of $f^{(k-1)}$ is an integer $r$. Let  $c$ be a generator of $\cM_K^r$. Note that
$$
|c|^{-1}\leq q_K \|f^{(k-1)} - f^{(k-1)}(0)\|\leq q_K \|f - f(0)\|
$$
 by Lemma \ref{1SP}.
 Let $b_i$ be a set of representatives in $\cO_K$ of $\cO_K/c\cO_K$, for $i=1,\ldots,|c|^{-1}$. Write
$$
f_{b_i,c,k}(t) = \frac{1}{c^k} f(b_i + ct).
$$
Then each of the $f^{(k-1)}_{b_i,c,k}$ is SP and satisfies
$|f_{b_i,c,k}^{(k-1)}|_{\SP}\geq 1$ by Lemma \ref{1SP} and the chain rule for differentiation.
 After a linear change of
variables and by the linearity of the integral we can write
$$
I(y) =
 \sum_{i=1}^{|c|^{-1}}  \int_{b_i+c \cO_K} \psi (y\cdot f(x))|dx| = |c| \sum_{i=1}^{|c|^{-1}}  \int_{\cO_K} \psi ( (c^ky) \cdot f_{b_i,c,k}(t))|dt|.
$$
If we abbreviate the $i$-th term as follows,
$$
I_i(y) := \int_{\cO_K}
\psi ((c^k y)\cdot f_{b_i,c,k}(t))|dt|,
$$
then
\begin{equation}\label{eq:|I(y)|2}
| I(y) | \leq  |c| \sum_{i=1}^{|c|^{-1}}  | I_i(y) |,
\end{equation}
or in words, $| I(y) |$ is bounded by the average value of the $| I_i(y) |$.

We now focus on the case that $k=1$. %, where we write $f_{b_i,c}$ for $f_{b_i,c,1}$.
 By Lemma \ref{oscSP}, for each $i$,
$$
| I_i(y) | \leq
 q_K^{-1} | f_{b_i,c,1} |_{\SP}^{-1}|cy|^{-1}\leq q_K^{-1}  |cy|^{-1}
$$
and thus
$$
|\ \int_{\cO_K} \psi (y\cdot f(x))|dx|\ | \leq q_K^{-1} |cy|^{-1}.
$$
We are done by Lemma \ref{1SP} in the case that $k=1$.
%Concluding, for $c_1$ we can take $\|f^{(1)}\| $.
\par

Finally fix $k\geq 2$ and suppose that the proposition is proved for all values up to $k-1$. So we
start from the condition that $|f^{(k)}|\geq 1$ on $\cO_K$.
 Recall that $f^{(k-1)}_{b_i,c,k}$ is SP  for each $i$ and satisfies
$|f_{b_i,c,k}^{(k-1)}|_{\SP}\geq 1$.
%Hence, by Lemma \ref{SP2alt}, if there exists no $d$ in $\cO_K$ with $f_{b_i,c,k}^{(k-1)}(d)=0$ then  $|f_{b_i,c,k}^{(k-1)}(x)|>1$ for all $x\in \cO_K$. In this case for our fixed $i$, we find by induction that
%$$
%| I_i(y)  |
% \leq c'_{k-1} |c^ky|^{-\frac{1}{k-1}},
%$$
%where $c'_{k-1}$ depends only on $k-1$, on $q_K$, and on
%$\| f_{b_i,c,k} - f_{b_i,c,k} (0)\| $. Note that
%$$
%1\leq \| f_{b_i,c,k} - f_{b_i,c,k} (0)\|  \leq c^{-k} \| f - f(0) \|  = q_K^k \| f - f(0) ||^{k+1}.
%$$
%%and that there are only finitely many values the Gauss norm can attain between any two given (strictly) positive numbers.
%Since we moreover are considering $y$ with $|y|>1$, we may as well bound
%\begin{equation}\label{eq:no:d}
%|I_i(y) |
% \leq \bar c_{k} |y|^{-\frac{1}{k}},
%\end{equation}
%where $\bar c_{k}$ depends only on $k$, on $q_K$, and on
%$||f - f(0)||$.
 Fix $i$ and suppose that $|f_{b_i,c,k}^{(k-1)} (d)|$ is minimal for some $d\in\cO_K$ among the values $|f_{b_i,c,k}^{(k-1)} (x)|$ for $x\in\cO_K$. Up to translating by $d$, we may suppose that $d=0$.
 Then, by Lemma
\ref{SP2alt},
\begin{equation}\label{bigi0}
|f_{b_i,c,k}^{(k-1)}(x)| \geq   |x|
\end{equation}
for all $x\in\cO_K$.
  Take a nonzero $\gamma\in\cO_K$. Partition $\cO_K$ into the ball
$$B_0:=\gamma \cO_K$$
 and $n$ balls of the form
$$
B_j:=d_{j}+ n_j \cO_K
$$
for $d_j$ with $|d_j|>|\gamma|$ and $n_j$ a generator of the ideal $d_j\cM_K$,
$j=1,\ldots,n$, and where necessarily $n= (q_K-1)\ord(\gamma)$.
 The ball $B_0$  will serve as a bad ball where we will use a trivial bound (namely the volume of $B_0$), while on the remaining $B_j$ we will use bounds coming from induction. At some point, we will optimize the choice of $\gamma$ for any given value of $y$ (which is similar to the proof of van der Corput's Lemma in \cite{Stein}). In this optimization, it is important that there are not too many parts $B_j$, which is indeed achieved by are choice of rather big radii $n_j$. %(this point is trivial in the real case since removing a bad sub-interval of an interval leaves one with at most two remaining good intervals).
Finally we will combine again the terms for all the $i$ by (\ref{eq:|I(y)|2}).
 We write by the linearity of the integral
\begin{equation}\label{Iij}
 I_i(y) = \sum_{j=0}^n  I_{ij}(y)
\end{equation}
with
$$
I_{ij}(y) := \int_{B_j} \psi (c^ky\cdot f_{b_i,c,k}(x))|dx|.
$$
Clearly % (namely since $|\psi (y\cdot f(x))|=1$),
$$
| I_{i0}(y)  |   \leq  \int_{B_0} | \psi (c^ky\cdot f_{b_i,c,k}(x)) | |dx|  = \int_{B_0}|dx| = |\gamma|.
 $$
 For $j=1,\ldots, n$ we can write, after a linear change of variables,
$$
I_{ij}(y) = |n_j |\int_{\cO_K} \psi
(c^k y \cdot g_{j}(t))|dt|  ,
$$
where
$$
g_{j}(t):= f_{b_i,c,k}(d_{j}+ n_j t).
$$
 By the definition of the SP-number, the $g_j^{(k-1)}$ are SP and by the chain rule $|g_j^{(k)}(t)|\geq |n_j^{k}|$ for all $t$ in $\cO_K$.  In fact,
the $g_j$ are even better than that, allowing us to use the induction hypothesis for each $j$. Indeed, by (\ref{bigi0}) and the chain rule one has  $|g_j^{(k-1)}(t)|\geq |n_j^{k-1} d_j| > 0 $ for all $t$ in $\cO_K$.
Note also that the Gauss norm of $g_j-g_j(0)$ is bounded by the Gauss norm of $(f-f(0))/c^k$, and that, in the case that $K$ has characteristic zero, the SP-number of $g_j^{(k-2)}$ is bounded by $\ord(p_K)$ by Corollary \ref{SPff'}.
  Therefore, we can use the
induction hypothesis in $k$ to $g_j$ to find
$$
| I_{ij}(y) | = | n_j | \cdot  |\ \int_{\cO_K} \psi (  c^k y \cdot g_{j}(t))|dt| \ |
 \leq  \widetilde c_{k-1}  \cdot  |d_j c^k y|^{-\frac{1}{k-1}} \leq c'_{k-1}  \cdot  |d_j y|^{-\frac{1}{k-1}},
$$
where $c_{k-1}'\geq \widetilde c_{k-1}|c|^{-\frac{k}{k-1}}$, and where $\widetilde c_{k-1}$ only depends on $k$, $q_K$, and $\|f - f(0) \| $. If $K$ has characteristic zero we can alternatively suppose that $\widetilde c_{k-1}$ only depends on $q_K$, $\ord(p_K)$, and $k$.
Hence, we may by Lemma \ref{1SP} suppose  that $c'_{k-1}$ only depends on $k$, $q_K$, and $\| f - f(0) \| $, or alternatively, if $K$ has characteristic zero, that $c'_{k-1}$ only depends on $k$, $q_K$, $\ord(p_K)$, and the SP-number of $f^{(k-1)}$.
 Summing up for $j=1,\ldots,n$ as in (\ref{Iij}) yields, still with
$n= (q_K-1)\ord(\gamma)$,
\begin{equation}\label{Ii(y)d}
| I_i(y) | \leq |\gamma| + c'_{k-1}
 |y|^{-\frac{1}{k-1}} \sum_{j=1}^n  |d_j |^{-{1/(k-1)}}.
\end{equation}
For each $\ell$ with $0\leq \ell<\ord(\gamma)$, there are exactly $q_K-1$ different $d_j$ with $\ord (d_j)=\ell$. Hence we can calculate:
$$
\sum_{j=1}^n  |d_j |^{-{1/(k-1)}} =
 (q_K-1)  \sum_{\ell\geq 0}^{\ord(\gamma)-1} (q_K^{1/(k-1)})^\ell
$$
$$
=(q_K-1) \frac{ |\gamma|^{-1/(k-1)} - 1}{ q_K^{1/(k-1)} - 1}
$$
$$
\leq  |\gamma|^{-1/(k-1)}   \frac{ q_K-1  }{ q_K^{1/(k-1)} - 1}
$$
%$$
%\leq |\gamma|^{-1/(k-1)} (k-1) q_K^{1/(k-1)}
%$$
Combining with (\ref{Ii(y)d}) yields
\begin{equation}\label{ygamma}
|I_i(y)  | \leq |\gamma| + c''_{k-1}
 |\gamma y|^{-\frac{1}{k-1}}
\end{equation}
 for some $c''_{k-1}$ only depending on $k$, $q_K$, and $\| f - f(0) \| $, resp.~in characteristic zero only depending on $k$, $q_K$, $\ord(p_K)$, and the SP-number of $f^{(k-1)}$.
 Recall that we are considering $y$ with $|y|\geq 1$.
Choose $\gamma$ in $\cO_K$
such that
\begin{equation}\label{gamma}
q_K^{-1} |y|^{-1/k} \leq  |\gamma| < |y|^{-1/k}.
\end{equation}
 Together with  (\ref{ygamma}) this gives
\begin{equation}\label{c'''k}
|I_i(y)  | \leq c'''_{k}
 |y|^{-\frac{1}{k}}
\end{equation}
 for some $c'''_{k}$ only depending on $k$, $q_K$, and $\| f - f(0)\| $, resp.~only depending on $k$, $q_K$, $\ord(p_K)$, and the SP-number of $f^{(k-1)}$.
 %(Note that such $\gamma$ exists since the conditions on $\gamma$ are equivalent with $q_K^{-1} |y|^{-1/k} %\leq  |\gamma| < |y|^{-1/k}$. )
Putting the bounds (\ref{c'''k}) in (\ref{eq:|I(y)|2}) yields  the desired bound for $|I(y)|$ in terms of some constant $c_k$ only depending on $k$, $q_K$, and $\| f - f(0)\| $, resp.~in characteristic zero only depending on $k$, $q_K$, $\ord(p_K)$, and the SP-number of $f^{(k-1)}$.
\end{proof}

%\begin{example}

%\end{example}

As a corollary of Proposition \ref{osc(k)}, we make a link between Weierstrass regularity of some derivative of $f$ and the conditions of the van der Corput Lemma \ref{osc(k)}, to find back a generalization of the main thrust (Lemma 3) of \cite{Rogers}, from which Rogers derives in a beautiful and direct way all principal results of \cite{Rogers}. Rogers  gives in Lemma 3 of \cite{Rogers}, in the case that $f$ is a polynomial over $\QQ_p$ and only treating the case $j=1$, explicit values for the $c_{m,\QQ_p}$ of Corollary \ref{general:Rogers}.

We still consider our fixed  $f$ in in $K \lla x\rra$ and the corresponding oscillating integral $I(y)$ as just above Proposition \ref{osc(k)}.
\begin{cor}\label{general:Rogers}
Suppose that  $f^{(j)}$ is (Weierstrass) regular of degree $m-j>0$ for some $j\geq 0$. If the characteristic of $K$ is zero, then there exists $c_{m,K}$, only depending on $m$, $q_K$, and $\ord(p_K)$, such that, for all nonzero $y\in K$,
$$
| I(y) | \leq c_{m,K} |y|^{-1/m}.
$$
%In fact, one can even take $c_{m,K}$ only depending on $m$, $q_K$, and the norm $|m!|$ of $m!= m\cdot(m-1)\cdot\ldots \cdot1 $ in $K$.
\end{cor}
\begin{proof}
Clearly on the one hand
$$
|f^{(m)}(x)| \geq |(m-j)!| \mbox{ for all } x\in\cO_K,
$$
and on the other hand, the SP-number of $f^{(m-1)}$ is zero. Now apply Proposition \ref{osc(k)}.
\end{proof}

%\begin{example}\label{badC1}
%
%
%
%\end{example}
%

%Hence, one can take $c_1^{\SP}= q_K^{-1}$, $c_2^{\SP}= 1 $, $c_3^{\SP}=  q_K^{-1/3}(1 + 2q_K)$ and so on.
%% $c_2^{\SP}= 3$, $c_3^{\SP}=  1+9q_K $ and so on.

\subsection{Several variables}\label{sec:sev}

Now that we have obtained the non-archimedean analogue of the van der Corput Lemma for analytic phases, we can grasp its rewards and develop the theory in great analogy to \cite[Sections 2 and 3, Chapter VIII]{Stein}. Note that in \cite{Cexp}, decay rates for higher dimensional non-archimedean Fourier transforms have been obtained for $L^1$-functions of a certain constructible nature, related to subanalytic functions. Here we will find more explicit decay rates, in a different setting as in \cite{Cexp} which is in some ways more general and in other ways more restricting.

From now on we will consider tuples of variables  $x=(x_1,\ldots,x_n)$, writing $K \lla x\rra$ for the collection of power series in the
variables $x$ over $K$ which converge on $\cO_K^n$, that is, those
power series $\sum_{i\in\NN^n} a_ix^i\in K[[x]]$ satisfying that
$|a_i|$ goes to zero when $|i|:=\sum_{j=1}^n i_j$ goes to infinity. Likewise, we write
$\cO_K \lla x\rra$ for power series in $K \lla x\rra$ which also lie
in $\cO_K[[x]]$ and for $f(x)\in K\lla x \rra $, we write $\| f\| $ for the
Gauss norm of $f$, which is $\sup_{i\in\NN^n}|a_i|$.

The following is the non-archimedean analogue of \cite[Proposition 5, Chapter VIII]{Stein} for analytic phase in the oscillating integral (where \cite{Stein} is for real, smooth phase); note that in our proposition the Gauss norm of  $f-f(0)$ plays the role of the $C^{k+1}$ norm of the phase in Proposition 5 of \cite[Chapter VIII]{Stein}.

\begin{prop}\label{p:several}
Let $f(x)$ be a power series in $K \lla x\rra$ in the variables $x=(x_1,\ldots,x_n)$.
 Suppose that for some multi-index $\alpha\in\NN^n$ with $|\alpha|>0$, one has
 $$
 |\partial^\alpha_x f(x)|\geq 1 \mbox{ for all }x\in\cO_K^n,
 $$
 where $|\alpha|=\sum_j\alpha_j$ and $\partial^\alpha_x f= (\prod_j \frac{\partial^{\alpha_j}}{\partial x_j^{\alpha_j}}) f$. Suppose also that the characteristic of $K$ is $>|\alpha|$.
Then, for all nonzero $y\in K$,
$$
|\ \int_{\cO_K^n} \psi (y\cdot f(x))|dx| \ | \leq  d_k
|y|^{-\frac{1}{k}},
$$
where $d_k$ only depends on $K$, $n$, $k=|\alpha|$, and on $\| f-f(0)\| $.
\end{prop}
\begin{proof}%[Proof of Proposition \ref{p:several}]
Consider the $K$-vector space $V_{k,n}(K)$ of homogeneous polynomials of degree $k$ over $K$ in the $n$ variables $x=(x_1,\ldots,x_n)$. By Lemma \ref{l:hom}, there are vectors $\xi_1,\ldots\xi_d$ in $K^n$ of length $1$ (that is, $|\xi_i|=1$) such that the homogeneous polynomials
$$
(\xi_i\cdot x)^k, \quad i=1,\ldots,d
$$
form a basis for this vector space, with $d$ the dimension of $V_{k,n}(K)$. Express the monomial $x^\alpha$ in this basis as
$$
x^\alpha = \sum_i e_i (\xi_i\cdot x)^k,\qquad e_i\in K.
$$
%one has $|e|\geq 1$ by the non-archimedean property of the norm on $K^d$, with $c=(e_i)_i$.
Then, for $x_0\in\cO_K^n$,
$$
1\leq |\partial^\alpha_x f(x_0)| = | \sum_i e_i (\xi_i\cdot \nabla)^k f(x_0) | \leq \max_i ( | e_i (\xi_i\cdot \nabla)^k f(x_0) | )
$$
and hence
$$
| (\xi_i\cdot \nabla)^k f(x_0)| \geq  | 1/e_i |
$$
for at least one $i$ with $e_i\not = 0$. Note that this implies that $\| e_i(f-f(0))\| \geq 1$.
Hence, for such  $i$,
$$
| (\xi_i\cdot \nabla)^k f(x)| \geq  | 1/e_i |
$$
for all $x$ in the ball $B(x_0):= x_0 +  c \cO_K^n $ around $x_0$ with $c\in \cM$ satisfying $|c|= \| e_i(f-f(0))\| ^{-1}q_K^{-1}$.  Define $g(z)$ as
$e_ic^{-k}f(x_0+cz)$ with $z=(z_1,\ldots,z_n)$. Then $g(z)$ is in $K \lla z\rra$  and satisfies
 $$
 |(\xi_i\cdot \nabla)^k g|\geq 1.
 $$
After a measure preserving affine change of variables on $K^n$ such that $x_1$ lies along $\xi_i$, we may suppose that $\xi_i=(1,0,\ldots,0)$, and thus that
 $$
 | (\partial^k/\partial z_1^k) g(z)|\geq 1 \quad \mbox{ for all } z\in \cO_K^n.
 $$
For each $a_2,\ldots,a_n$ in $\cO_K$, the Gauss norm of  $g(t,a_2,\ldots,a_n)-g(0,a_2,\ldots,a_n)$ (where this power series lies in $K \lla t\rra$), is bounded by   $\|g-g(0)\|$. Hence, by Proposition \ref{osc(k)}, we find
$$
|\ \int_{B(x_0)} \psi (y\cdot f(x))|dx| =
$$
$$
|c^{n}|\cdot|\ \int_{\cO_K^n} \psi ( \frac{c^k y}{e_i} \cdot g(z))|dz| \ | = |c^{n}|\cdot |\ \int_{\cO_K^{n-1}}( \int_{\cO_K} \psi ( \frac{c^k y}{e_i}\cdot g(z)) |dz_1|   ) |dz_2\ldots dz_n| \ |
$$
$$
\leq |c^{n}|\cdot |\ \int_{\cO_K^{n-1}} c_k |c|^{-1} | y|^{-\frac{1}{k}} |dz_2\ldots dz_n| \ | = c_k |c|^{n-1} | y|^{-\frac{1}{k}}
$$
where $c_k$ only depends on $k$, $n$, $K$, and on the Gauss norm of
$g-g(0)$.
%Thus,
%$$
%|\ \int_{B(x_0)} \psi (y\cdot f(x))|dx| = |c^{n}| \cdot |\  \int_{\cO_K^n} \psi ( (e_i^{-1} c^k y)\cdot g(z))|dz| \ |
%$$
Since the Gauss norm of $g-g(0)$ is bounded by $|e_i c^{-k}|\cdot \| (f-f(0))\| $, and since
$$
|\ \int_{\cO_K^n} \psi (y\cdot f(x))|dx| \ | \leq \sum_{i=1}^{|c|^{-n}}  |\ \int_{B(b_i)} \psi (y\cdot f(x))|dx| \ |
$$
for any set $b_i$ of representatives of $\cO_K^n$ modulo $c\cO_K^n$, we are done.
\end{proof}

The following elementary lemma and its proof are a close adaptation of \cite[Chapter VIII, 2.2.1]{Stein} to a slightly more general setting.

\begin{lem}\label{l:hom}
Let $k>0$ be an integer and $x=(x_1,\ldots,x_n)$ variables. Let $L$ be an infinite field of characteristic $>k$.
Then the polynomials of the form
$$
(\xi\cdot x)^k,\quad \xi\in L^n,
$$
where $\xi\cdot x=\sum_i \xi_i x_i$, span the $L$-vector space $V_{k,n}(L)$ of homogeneous polynomials of degree $k$ in $x$ over $L$.
\end{lem}
\begin{proof}
On this vector space $V_{k,n}(L)$, consider the inner product (that is, bi-linear mapping to $L$)
$$
\langle P,Q \rangle = \alpha! a_\alpha b_\alpha,
$$
where $P(x)=\sum a_\alpha x^\alpha$ and $Q(x)=\sum b_\alpha x^\alpha$ and where $\alpha! = \prod_j (\alpha_j!)$. Note that
$$
\langle P,Q \rangle = (Q(\partial/\partial x) )(P),
$$
where the polynomials are derivated formally and where $\partial/\partial x = (\partial/\partial x_1,\ldots,\partial/\partial x_n)$.
Thus, if $P$ were orthogonal to all the polynomials of the form $(\xi\cdot x)^k$, then
$$
(\xi\cdot\nabla)^k(P)= 0,\quad \mbox{ for all } \xi\in L^n.
$$
In other words,
\begin{equation*}
(\partial/\partial t )^k P(t\xi)=0 \quad \mbox{ for all } \xi\in L^n,
\end{equation*}
which can happen only if $P(x)$ is the zero polynomial.
Indeed,
$$
(\partial/\partial t )^k P(t\xi) = (\partial/\partial t )^k (t^kP(\xi)) = k! P(\xi).
$$
From this we can draw our conclusions. Suppose that the space spanned by the $(\xi\cdot x)^k$ has strictly smaller dimension than $V_{k,n}(L)$. Let $P_1,\ldots,P_{d'}$ be a basis for this span. But then
$$
\bigcap_{j=1,\ldots,d'} P_j^{\bot},
$$
where $P_j^{\bot}=\{Q\in V_{k,n}\mid P_j\cdot Q = 0\}$,
has dimension $>0$, since this intersection is the solution set of $d'$ homogeneous linear equations on $V_{k,n}(L)$. We are done by contradiction.

\end{proof}

\begin{remark}
A complex valued $C^\infty$ function $h$ with compact support defined on an open subset of $K^n$ is automatically locally constant, and it is constant on each ball in a finite partition of the support of $h$ into balls. Hence, we simplify notation by working with characteristic functions of balls to serve, for example, as amplitudes, instead of with complex valued $C^\infty$ functions with compact support (which are the so-called Schwartz-Bruhat functions). The adaptation in the following theorem with an amplitude which is a Schwartz-Bruhat function is trivial to make.
We will simplify likewise in section \ref{sec:Fourier}.
\end{remark}

The following result for mappings is closely related to Theorem \ref{thm:Fourier:finitetype} below.

\begin{prop}[Mappings in several variables]\label{thm:multi/several}
Let $f_1,\ldots,f_n$ be power series in $K \lla x\rra$ in the variables $x=(x_1,\ldots,x_d)$ and let $k\geq 1$ be an integer.
 Suppose that for each $v\in K^n$ of length $1$ and for each $x\in K^d$, there exists some multi-index $\alpha\in\NN^d$ with $ k \geq |\alpha|>0$ and
 $$
 |v\cdot(\partial^\alpha_x f(x)) |\geq 1,
 $$
 where $|\alpha|=\sum_j\alpha_j$ and $\partial^\alpha_x f= ( (\prod_j \frac{\partial^{\alpha_j} }{\partial x_j^{\alpha_j}}) f_i )_i$. Suppose also that the characteristic of $K$ is $>k$.
Then there exists a constant $c$ such that
$$
|\ \int_{\cO_K^n} \psi (y\cdot f(x))|dx| \ | \leq  c
|y|^{-\frac{1}{k}}
$$
for all nonzero $y\in K^n$. Moreover, $c$ only depends on $K$, $n$, $k$, and on $\| f-f(0)\| $.
\end{prop}

\begin{proof}
Write $y$ as $\lambda v$, where $\lambda\in K^\times$ and $v\in K^n$ with $|v|=1$. Then, for each such $v$,   $$
|\ \int_{\cO_K^n} \psi (\lambda v \cdot f(x))|dx| \ | \leq  d_k
| \lambda |^{-\frac{1}{k}}
$$
for some $d_k$ only depending on $K$, $n$, $k$, and on $\| v\cdot f - v\cdot f(0)\| $, by Proposition \ref{p:several}. The Gauss norm $\| v\cdot f - v\cdot f(0)\| $ takes only finitely many values when $v$ varies over all vectors of length $1$, since this Gauss norm varies continuously in $v$ and $v$ runs over a compact. Even more, $\| v\cdot f - v\cdot f(0)\| $ takes only values between $1$ and $\| f-f(0)\| $. Hence we are done.
\end{proof}

\subsection{$K$-analytic manifolds of finite type and singular Fourier transforms}\label{sec:Fourier}

For an open $X\subset K^n$, $n\geq 0$, a function $f:X\to K^m$ is called $K$-analytic if there is an open cover of $X$ such that for each open $U$ in the cover the restriction of $f$ to $U$ is given by $m$ power series which converge on $U$.
Call a subset $M$ of $K^n$ for some $n\geq 0$ a $K$-analytic manifold of dimension $d$ if  there exists an open cover of $M$ such that for each open $U$ in the cover there exists a coordinate projection $p_U:K^n\to K^d$ onto $d$ of the $n$ standard coordinates on $K^n$, such that $p_U$ induces a bijection $U\to U'$ onto an open $U'\subset K^d$, and there exist $K$-analytic functions $f_1,\ldots,f_n:U'\to K$ such that $f=(f_1,\ldots,f_n)$ is the inverse map of $p_U$ on $U$ . Note that for a $K$-analytic manifold $M$, the open cover and the coordinate projections $p_U$ can be taken such that the $p_U$ are isometries. Then the induced volume $\mu_M$ on $M$ is by definition the pull-back of the standard normalized Haar measure on $K^d$ via the isometries $p_U$. By a $K$-analytic manifold we mean a $K$-analytic manifold of some dimension $d$.

\par
Say that a $K$-analytic manifold $M\subset K^n$ is of finite type at $x_0\in M$, if, for each hyperplane $H$ in $K^n$ containing $x_0$, and for each open $U$ in $M$ around $x_0$, one has that $U$ is not contained in $H$.\footnote{This includes the case that $M$ is open in $K^n$.}
 Equivalently, for $f=(f_1,\ldots,f_n)$ analytic coordinates on $M$ mapping $U'$ to $U\subset M$ as in the definition  of $K$-analytic manifolds and such that $x_0\in U$, one says that $M$ is of finite type at $x_0\in M$ if there exists $k>0$ such that for each nonzero $v\in K^n$ there exists $\alpha\in\NN^d$ with $k\geq |\alpha| := \sum_{i=1}^d \alpha_i >0 $ and such that
$$
v \cdot (\partial^\alpha f / \partial x^\alpha )(x_0) \not =0
$$
and the least value for $k$ with this property is called the type of $M$ at $x_0$.
 If $M$ is of finite type at all its points, then we call $M$ of finite type. If there exists $k$ such that $M$ is of type $\leq k$ at each of its points, then the least such integer $k$ is called the type of $M$.

\par
Now let $M\subset K^n$ be a $K$-analytic manifold of finite type, and let $S$ be a compact open in $M$. Then $S$ is a $K$-analytic manifold of type $k$ for some $k>0$.  Let $\mu_S$ be the induced measure on $S$. The Fourier transform of $\mu_S$ is defined for $y\in K^n$ by
$$
\widehat \mu_S (y) := \int_{x\in S} \psi ( y\cdot x  ) d\mu_S(x),
$$
and thus it is a complex valued function on $K^n$.

The following Theorem is the non-archimedean analogue of \cite[Theorem 2, Chapter VIII]{Stein}.

\begin{theorem}\label{thm:Fourier:finitetype} Suppose that the characteristic of $K$ is $>|\alpha|$.
Then
$$
\lim_{|y|\to \infty} \widehat \mu_S (y) =0,
$$
and, more precisely,  there exists a constant $c$ such that
$$
|\widehat \mu_S (y)| \leq c |y|^{\frac{-1}{k}}\ \mbox{ for all nonzero $y\in K^n$.}
$$
\end{theorem}
\begin{proof}
By using finitely many charts with analytic isometries for the maps $p_U$ as in the definition of $K$-analytic manifolds given above, the theorem is translated into a finite sum of integrals as treated in Proposition \ref{thm:multi/several}.
\end{proof}

\section{Restriction of the Fourier transform}

The above non-archimedean van der Corput Lemma allows us to develop the theory in great analogy to what follows on the real van der Corput Lemma in \cite{Stein} from Chapter VIII on. We will only implement a non-archimedean analogue of the important restriction result by E.~Stein \cite{Stein} and C.~Fefferman \cite{Fefferman71}, namely in the form of Theorem 3 of \cite[Section 4, Chapter VIII]{Stein}, as it is an ingenious and rather recent application of van der Corput's Lemma. In fact, we will stay very close to loc.~cit., sometimes transcribing rather directly from the real case to the non-archimedean case.

To sketch some context we base ourselves on the introduction from \cite[Section 4, Chapter VIII]{Stein}.
%\begin{quote}
The Fourier transform of an $L^1(K^n)$-function is a continuous function, and hence is defined everywhere on $K^n$. On the other hand, the Fourier transform of an $L^2$ function is itself no better than an $L^2$-function, and so can be defined only almost everywhere, and is thus completely arbitrary on a set of measure zero. In addition, when $1<p\leq 2$, the classical Hausdorff-Young theorem allows one to realize the Fourier transform of an $L^p$ function as an element of $L^q(K^n)$, $1/p+1/q=1$, and so, at first sight, is determined only almost everywhere.
 In view of this, it is a remarkable discovery by E.~Stein and C.~Fefferman in the real case and adapted here to the non-archimedean case, that when $n\geq 2$ and $M$ is a submanifold of $K^n$ that has appropriate curvature, there is a $p_0=p_0(M)$, with $1<p_0<2$, so that every function in $L^p(K^n)$, for any $p$ with $1\leq p<p_0$, has a Fourier transform that has a well-defined restriction to $M$.
%That this phenomenon was observed so late in the development of the subject is an indication of the slowness of our progress in understanding the %genuinly $n$-dimensional aspects of Fourier analysis.
%\end{quote}

\subsection{}

Let us make the notion of restriction of the Fourier transform precise.

Suppose that $M\subset K^n$ is a $K$-analytic manifold with induced measure $\mu_M$. Say that the $L^p$ restriction property is valid for $M$ if there exists a $q=q(p)$ so that the inequality
\begin{equation}\label{eq:restr}
\big( \int_{M_0} |\widehat f (\xi)|^q \mu_M(\xi)   \big)^{1/q} \leq A_{p,q}(M_0) \cdot  \| f\| _{L^p}
\end{equation}
holds for each Schwartz-Bruhat function $f$ on $K^n$ with Fourier transform $\widehat f$, whenever $M_0$ is a compact open subset of $M$.
Because the space of Schwartz-Bruhat functions is dense in $L^p$ we can, when (\ref{eq:restr}) holds, define $\widehat f$ on $M$ (almost everywhere for $\mu_M$), for each $f$ in $L^p(K^n)$.

\begin{theorem}\label{thm:restr}
Let $M\subset K^n$ be a $K$-analytic manifold of type $k$, and suppose that the characteristic of $K$ is either $0$ or $>k$. Then there exists a $p_0$ depending on $M$ and with $p_0>1$, such that $M$ has the $L^p$ restriction property (\ref{eq:restr}) for all $p$ with $1\leq p\leq p_0$ and $q=2$.
\end{theorem}
\begin{remark}
The analysis being similar to the analysis of \cite[Section 4, Chapter VIII]{Stein}, we will get the same value
$$
p_0 = \frac{2nk}{2nk - 1 }
$$
as in Theorem 3 of \cite[Section 4, Chapter VIII]{Stein}.
\end{remark}

\begin{proof}
It will suffice to prove that, for compact open $M_0\subset M$,
$$
\big( \int_{M_0} |\widehat f (\xi)|^2 \mu_M(\xi)   \big)^{1/2} \leq A  \cdot  \| f\| _{L^p(K^n)}
$$
for any Schwartz-Bruhat function $f$ on $K^n$.
Define $\mu$ as $\chi_{M_0} \mu_M$ with $\chi_{M_0}$ the characteristic function of $M_0$. Consider the operator $R$ on Schwartz-Bruhat functions, where $R f (\xi)$ is defined for $\xi\in M$ by the Fourier transform
$$
Rf(\xi) = \int_{K^n} \psi ( x\cdot \xi) f(x) |dx|
$$
The question then is whether $R$ can be seen as a bounded mapping from $L^p(K^n)$ to $L^2(M,\mu)$, and, in studying this, we consider also its formal adjoint $R^*$, given for $f$ and $x\in K^n$ by
$$
R^* f (x) = \int_{M} \psi ( - x\cdot \xi) f(\xi) \mu(\xi).
$$
We have
$$
\langle R f , Rf  \rangle_{L^2(M,\mu)} = \langle R^* R f , f \rangle_{L^2(K^n)},
$$
so to prove that
$$
R:L^p(K^n)\to L^2(M,\mu)
$$
is bounded, it suffices, by H\"older's inequality, to see that
$$
R^*R : L^p(K^n)\to L^{p'}(K^n)
$$
is bounded, where $p'$ is the exponent conjugate to $p$. One sees that
$$
(R^*R f )(x) = \int_{K^n}\int_M \psi ( \xi\cdot (y- x)   ) \mu(\xi) f(y) |dy|,
$$
so $(R^*R f )(x) = (f \ast L ) (x)$ with
$$
L(x) = \widehat \mu ( - x),
$$
where $\widehat \mu$ is as in section \ref{sec:Fourier}.
By Theorem \ref{thm:Fourier:finitetype}, we have   for all nonzero $x\in K^n$
$$
| L(x) | \leq A |x|^{\frac{-1}{k}}
$$
 for some $A$ and clearly $L$ is bounded, thus there exists a constant $A'$ such that for nonzero $x$
$$
| L(x) | \leq A' |x|^{-\gamma}, \quad \mbox{whenever } 0\leq \gamma\leq 1/k.
$$
By the theorem of fractional integration (see the Hardy-Littlewood-Sobolev inequality below), the operator
 $
f\mapsto f\ast (|x|^{-\gamma}  )
 $
is bounded from $L^p(K^n)$ to $L^q(K^n)$, whenever $1<p<q<\infty$ and $1/q = 1/p -1 + \gamma/n$. Then if $q=p'$, we have $1/q = 1 - 1/p$, so the relation among the exponents becomes $2 - 2/p=\gamma / n$, and the restriction $0\leq \gamma\leq 1/k$ becomes $1\leq p\leq 2nk/(2nk - 1)$, completing the proof of the theorem (since the case $p=1$ is trivial).
\end{proof}

\subsubsection{Hardy-Littlewood-Sobolev inequality}

In the real set-up, there are many proofs for the Hardy-Littlewood-Sobolev inequality, which in the non-archimedean case reads as the inequality
\begin{equation}\label{eq:HLS}
\|f\ast (|y|^{-\gamma}  ) \|_{L^q(K^n)} \leq A_{p,q}  \|f  \|_{L^p(K^n)}
\end{equation}
for
\begin{equation}\label{pqn}
0<\gamma<n,\ 1<p<q<\infty,\ \mbox{and } \frac{1}{q} = \frac{1}{p} - \frac{n-\gamma}{n},
\end{equation}
where we have written $|y|^{-\gamma}$ for the function $y\mapsto |y|^{-\gamma}$ on $K^n\setminus\{0\}$, extended trivially on $0$.
We will work out the non-archimedean version of the proof given in \cite[Section 4.2, Chapter VIII]{Stein}, which is based on Hedberg's proof in \cite{Hedberg72}.
First define, for any complex valued function $f$ on $K^n$,  the maximal function
$$
(M f) (x) := \sup_{x\in B} \frac{1}{\rm Vol (B)} \int_B |f(y)| |dy|,
$$
where the supremum is taken over all balls in $K^n$ containing $x$, ${\rm Vol (B)}$ denotes the volume of $B$, and where a ball is any subset of $K^n$ of the form
$$
a+b\cdot \cO_K^n
$$
for $a\in K^n$ and nonzero $b$ in $K$. (Note that   the distinction made in the real case between the maximal function and the so-called sharp maximal function, cf.~$M$ and $\tilde M$ on page 13 of \cite{Stein}, is irrelevant here since they coincide in the non-archimedean case.) Recall the Hardy-Littlewood-Wiener maximal theorem, with literally the same proof as in the archimedean case of  \cite[Theorem 1, Chapter 1]{Stein}, namely that for $f\in L^p(K^n)$, $1<p\leq \infty$, one has $M(f)\in L^p(K^n)$ and
\begin{equation}\label{maximalfunction}
\|M(f) \|_{L^p(K^n)} \leq A_p \| f \|_{L^p(K^n)}
\end{equation}
where the constant $A_p$ only depends on $K$, $n$, and $p$.
Now, write
$$
(f\ast (|y|^{-\gamma} ) (x) = \int_{K^n} f(x-y) |y|^{-\gamma} |dy| = \int_{|y|<R} + \int_{|y|\geq R}
$$
for $R>0$.

Note that for a characteristic function $\chi_B$ of a ball $B$ containing zero, by the definition of $M$, one has
$$
| (f\ast \chi_B)(x) | \leq (Mf)(x) \cdot {\rm Vol (B)}.
$$
Since one can approximate the function $y\mapsto |y|^{-\gamma}\chi_{B_R}(y)$ on $K^n$, with $B_R$ the ball given by $|y|<R$, by finite expressions of the form
$$
\sum_{i=1}^N a_i\chi_{B_i},
$$
with the $B_i$ balls around $0$,
it follows by the Lebesgue monotone convergence theorem that our first integral, which can be rewritten as the convolution
$$
  \int_{K^n} f(x-y) \big(|y|^{-\gamma}\chi_{B_R}(y)\big) |dy|,
$$
is bounded in norm by
$$
(Mf)(x)\cdot \int_{|y|<R} |y|^{-\gamma} |dy| \leq  \delta R^{n-\gamma}\cdot (Mf)(x),
$$
for some constant $\delta$.
By H\"older's inequality, the second integral
$$
\int_{|y|\geq R} f(x-y) |y|^{-\gamma} |dy|
$$
is dominated by
$$
\|f \|_{L^p(K^n)} \cdot \| |y|^{-\gamma} \cdot(1 - \chi_{B_R} )   \|_{L^{p'}(K^n)}.
$$
Now $|y|^{-\gamma} \cdot(1 - \chi_{B_R} )$ is in $L^{p'}(K^n)$ when
$-\gamma p'<-n $ and, in view of (\ref{pqn}),
$$
\gamma p' - n = \frac{np'}{q}>0.
$$
Thus
$$
\| |y|^{-\gamma} \cdot(1 - \chi_{B_R} )   \|_{L^{p'}(K^n)} \leq e R^{-n/q}
$$
for some constant $e$.
Summing the two integrals, we have
$$
|(f\ast (|y|^{-\gamma} ) (x)| \leq A\big( (Mf(x)R^{n-\gamma} +    \|f \|_{L^p(K^n)}   R^{-n/q}  \big)
$$
for some constant $A$.
Choose $R$ so that both terms on the right side are equal, that is,
$$
\frac{(Mf)(x)}{\|f \|_{L^p(K^n)}} = R^{-n+\gamma-n/q} = R^{-n/p}.
$$
Substituting this in the above gives
$$
|(f\ast (|y|^{-\gamma} ) (x)| \leq 2\cdot A \cdot (Mf(x))^{p/q} \cdot    \|f \|_{L^p(K^n)}^{1-p/q}.
$$
The inequality (\ref{eq:HLS}) then follows from the usual $L^p$ inequality for the maximal operator $M$, namely the inequality (\ref{maximalfunction}) of the Hardy-Littlewood-Wiener maximal theorem.

%\subsection*{Acknowledgment} I would like to thank warmly A.~Valette for
%inviting me to work on the question addressed in Theorem \ref{thm:Fourier:finitetype}.

\bibliographystyle{amsplain}

\bibliography{anbib}

\end{document}